\documentclass[12pt,oneside]{article}

\usepackage{setspace}
\usepackage{graphicx}
\usepackage[scale=0.7]{geometry}
\usepackage[utf8]{inputenc}
\usepackage{color}
\usepackage{tikz}
\usetikzlibrary{calc}

\usepackage{hyperref}
\hypersetup{breaklinks,colorlinks=true,linkcolor=[rgb]{0.45,0.0,0.0},citecolor=[rgb]{0,0,0.45}}

\usepackage{amsthm}
\usepackage{amsmath}
\usepackage{amssymb}
\usepackage{units}

\newtheorem{theorem}{Theorem}
\newtheorem{lemma}[theorem]{Lemma}
\newtheorem{claim}[theorem]{Claim}

\newtheorem{remark}[theorem]{Remark}

\newcommand{\R}{\mathbb{R}}
\newcommand{\I}{\textrm{I}}
\newcommand{\cir}{\mathbb{S}^1}
\newcommand{\Ss}{\mathbb{S}}

\newcommand{\rangeof}[1]{\underline{\smash{#1}}}
\newcommand{\curvemap}{\gamma}
\newcommand{\curve}{\underline{\smash{\gamma}}}
\newcommand{\pathimap}[1]{\lambda_i}
\newcommand{\pathi}[1]{\underline{\smash{\lambda_{#1}}}}

\newcommand{\mb}{\mathbb}

\newcommand{\eps}{\varepsilon}
\newcommand{\abs}[1]{\lvert #1 \rvert}
\newcommand{\email}[1]{\texttt{#1}}

\renewcommand{\phi}{\varphi}
\renewcommand{\ell}{l}
\newcommand{\x}{\text{x}}
\newcommand{\y}{\text{y}}
\newcommand{\z}{\text{z}}

\title{Shadows of a Closed Curve}

\begin{document}


\begin{center}
{\Large \noindent
Shadows of a Closed Curve}

\bigskip

\begin{spacing}{1.5}
{\large \noindent
Michael~Gene~Dobbins\textsuperscript{1} \quad 
Heuna~Kim\textsuperscript{2} \quad 
Luis~Montejano\textsuperscript{3} \quad 
Edgardo~Roldán-Pensado\textsuperscript{4}
}
\end{spacing}

\smallskip

\begin{minipage}{0.9\textwidth}
\raggedright \footnotesize \singlespacing

\noindent
\llap{\textsuperscript{1}}Department of Mathematical Sciences, Binghamton University (SUNY), Binghamton, \newline New York, USA. 
\email{mdobbins@binghamton.edu} 
\newline
\llap{\textsuperscript{2}}Institute of Computer Science, Freie Universit\"at Berlin, Berlin, Germany. 
\email{heunak@mi.fu-berlin.de} 
\newline
\llap{\textsuperscript{3}}Instituto de Matem\'aticas, Universidad Nacion\'al Aut\'onoma de M\'exico, Juriquilla, Mexico. 
\email{luis@matem.unam.mx} 
\newline
\llap{\textsuperscript{4}}Centro de Ciencias Matem\'aticas, Universidad Nacion\'al Aut\'onoma de M\'exico, Morelia, Mexico.
\email{e.roldan@im.unam.mx}
\end{minipage}

\end{center}

\begin{abstract}
A shadow of a geometric object $A$ in a given direction $v$ is the orthogonal projection of $A$ on the hyperplane orthogonal to $v$. 
We show that any topological embedding of a circle into Euclidean $d$-space can have at most two shadows that are simple paths in linearly independent directions. 
The proof is topological and uses an analog of basic properties of degree of maps on a circle to relations on a circle. 
This extends a previous result which dealt with the case $d=3$.
 \end{abstract}

\section{Introduction}

Given a set $A$ in $\R^d$, we define the $i$-th coordinate shadow of $A$ as the image of $A$ 
by the orthogonal projection to the coordinate hyperplane $\{(x_1,\dots,x_d) \in\R^d : x_i = 0\}$.
Suppose we want to draw a closed curve in $\R^d$ so as to maximize the number of shadows that are paths.
It is easy to see that two shadows can be paths. Just consider the unit circle in a coordinate plane 
$A = \{(x_1,x_2,0\dots,0) \in\R^d:x_1^2+x_2^2=1\}$.  The $1$-st and $2$-nd coordinate shadows of $A$ are paths, but all others are circles.
We show that this is the best that can be done.

\begin{theorem}[version 1]\label{thm:curve}
A simple closed curve in $\R^d$ has at most two coordinate shadows that are simple paths.
\end{theorem}

By considering a curve up to linear transformations, Theorem~\ref{thm:curve} can be restated as follows:

\addtocounter{theorem}{-1}
\begin{theorem}[version 2]\label{thm:linearindep}
{For any simple closed curve $\curve$ in $\R^d$, it is not possible to project $\curve$ in three linearly independent directions such that the image by each projection is a simple path.}
\end{theorem}

Coordinate shadows are a common and effective tool for visualizing and analyzing geometric objects
in high-dimensional space.  For example, orthogonal projections are
used in classical methods for data compression \cite[Chapter 4.26.]{salomon} and dimension
reduction \cite{pearson}. 

Trying to describe topological properties of a set $A$ using topological properties of the coordinate shadows of $A$ might seem futile at first glance, because so much information about the set is lost, and also because coordinate shadows are a very geometric feature that depend delicately on a choice coordinates.   
Our result, however, alludes to a topological relation between a set and its coordinate shadows, and provide an early step toward answering the following more general inquiry.

\begin{quote}
Given an embedding of a topological space $A$ in some Euclidean space of higher dimension,
what does the topology of its shadows tell us about the topology of $A$?
\end{quote}

This is in the spirit of tomography, which studies how a set $A$ can be reconstructed from the volume of the intersection of $A$ with lower dimensional spaces (the Radon transform of $A$).
This question can be seen as an extreme case of sparse sampling in tomography where the information available is restricted to the support function of the Radon transform along lines in $d$ linearly independent directions \cite{gardner1995geometric}. 

\subsection{Background}

This problem was motivated by the following question asked by H. W. Lenstra. 
\begin{quote}
Is there a simple closed curve in $3$-space such that all three of its coordinate shadows are trees? 
\end{quote}
The original motivation for Lenstra’s question was Oskar’s puzzle cube, three
mutually orthogonal rods that pass though slits in the sides of a hollow cube.
The rods are joined at a common point, and the slits in the sides of the cube
comprise three mazes. To move the rods to a desired configuration, all three of these
mazes must be solved simultaneously. Lenstra originally asked if the three mazes
could be designed so that the point where the three rods meet can move along a
trajectory that returns to its starting position without backtracking, thus tracing
a closed curve. Of course, none of the three mazes in the sides of the cube can
contain a closed curve individually, since that would result in a side of the cube being disconnected.

An affirmative answer to this question was given by J. R. Rickard \cite[p.112]{Win2007} (see Figure \ref{fig:trees})
and several such curves were later shown to exist (e.g. \cite{Gou2012}). 
More about the history of this problem can be found in the book \emph{Mathematical Mind-Benders} 
by Peter Winkler \cite[p. 118]{Win2007} under the name of \emph{Curve and Three Shadows}. 
In fact, the cover of this book shows Rickard's curve.

\begin{figure}
\centering
\includegraphics[width=0.5\textwidth]{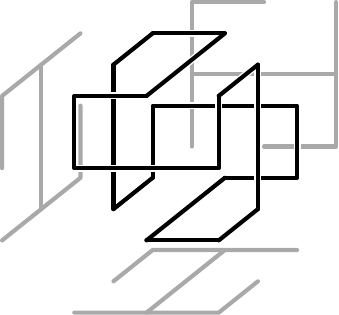}
\caption{A simple closed curve in $3$-space such that all of 
its shadows are trees.}
\label{fig:trees}
\end{figure}

A variant of this question is whether 
there is a simple closed curve in $3$-space such that its three coordinate shadows are simple paths. 
This was asked in CCCG 2007 \cite{DO2008} and in the 2012 Mathematics Research Communities workshop.
It has been shown that the answer is negative in~\cite{BDCD+2015}.
On the other hand, the same paper showed that there does exist a simple path in $3$-space 
such that all three of its shadows are simple closed curves~\cite{BDCD+2015}.
Our result is a generalization of the question asked in CCCG 2007 to any $d$-space.

As part of the proof, we employ a method that was used by Goodman, Pach, and Yap to show that two mountain climbers starting at sea level on opposite sides of a mountain (in the plane) can climb the mountain so that both climbers are always at equal altitude while making their way to the top \cite{goodman1989mountain}.

\subsection{Organization and structure of proof}

We will prove Theorem~\ref{thm:curve} by contradiction.  
Given a simple closed curve $\curve$ with three coordinate shadows that are each paths, 
we show in Section~\ref{sec:endpoints} that there cannot be a point $q_0 \in \curve$ that projects to an endpoint of each shadow.  
Then, in Section~\ref{sec:proof}, we show that such a point $q_0$ must exist. 
To show that $q_0$ exists, we define a relation on the curve such that any fixed point of the relation would project to an endpoint of each shadow.  Finally, we show that this relation does indeed have a fixed point.
Before getting into the proof that $q_0$ must exist, we prove a special case of Theorem \ref{thm:curve} in Section~\ref{sec:intuition} to provide geometric intuition.  For this, we use the fact that a degree ${-}1$ map from the circle to itself has a fixed point. 


\subsection{Notation and Terminology}

Let $\{x_i{=}c\}$ denote the hyperplane $\{(x_1,\dots,x_d)\in \R^d : x_i =c\}$, and let $\I$ denote the closed unit interval, 
$\cir$ will be the circle, and $\mb{T}^2 = \cir \times \cir$ will be the torus.  
Let $e_1,\dots,e_d$ denote the standard basis vectors in $\R^d$, and 
$e_1^*,\dots,e_d^* : \R^d \to \R$ denote the dual coordinate functions. 
Let $\pi_i : \R^d \to \{x_i{=}0\}$ be the orthogonal projection to the $i$-th coordinate hyperplane.
For a function $\phi: X \to Y$, let $\rangeof{\phi} = \phi(X)$ denote
the range of $\phi$.  Let $X^\circ$ denote the interior of a set.

Here we define a graph to be a $1$-dimensional cell complex.  That is,
a graph is a topological space consisting of a set of
vertices and edges, where an edge between a pairs of vertices $v,w$ is
given by a homeomorphic copy of the unit interval with endpoints
$v,w$, whose interior is disjoint from all other edges. 
Here it is enough to consider simple graphs, but allowing graphs to have loops or multiple edges in general will have no impact on our arguments. 
Generally we will be interested in graphs that are embedded in a torus where every vertex has degree $1$ or $2$. 

To avoid confusions due to multiple meanings of degree, we refer to the degree of a map from the circle to itself as topological degree and the degree of a vertex in a graph as graphical degree.

\section{Endpoints of three shadows}\label{sec:endpoints}

We will prove Theorem \ref{thm:curve} by contradiction using the following lemma.

\begin{lemma}\label{lem:endpoints}
If $\curve$ is a simple closed curve in $\R^d$ with three coordinate shadows that are each a path, then there cannot be a point $q_0 \in \curve$ that is projected to an endpoint of each of the three paths.
\end{lemma}

\begin{proof}
Suppose the lemma is false and let $\curvemap:\cir \to \R^d$ be a simple closed curve with a point $q_0 \in \curve$ 
such that $\pi_i(\curve)$ is a path having $\pi_i(q_0)$ as an endpoint for $i=1,2,3$. 

\begin{figure}
\centering
 \includegraphics{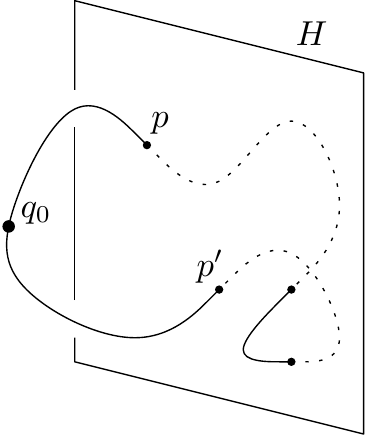}
\caption{The points $p$ and $p'$.}
\label{fig:plane}
\end{figure}

The curve $\curve$ cannot be contained in a hyperplane that is perpendicular to $e_1$; otherwise $\pi_1(\curve)$ would be a translate of $\curve$, but $\pi_1(\curve)$ is a path. 
Therefore, there exists a hyperplane $H = \{x_1 {=} c\}$ that intersects $\curve$ at multiple points, but does not contain $q_0$. 
Let $p$ and $p'$ be the endpoints of the connected component of $\curve\setminus H$ that contains $q_0$ (see Figure \ref{fig:plane}). 
Since $H$ intersects $\curve$ in multiple points, $p \neq p'$. 
Let $\phi : \I \to \curve$ be a parameterization of $\curve$ that starts at $q_0$ and passes through $p$ before $p'$, and let $\phi'(x) = \phi(1-x)$ parameterize $\curve$ in the opposite direction.  
And, let $\lambda_i : \I \to \pi_i(\curve)$ be a parameterization of the $i$th coordinate shadow that starts at the endpoint $\pi_i(q_0)$.
The first coordinate of $\phi$ attains the value $c$ for the first time at $p$, 
\[ p = \phi\left( \inf\{ t \in \I : \langle e_1, \phi(t)\rangle = c \}  \right). \]
Therefore $\pi_2(p)$ is the point where $\pi_2 \circ \phi$ attains the value $c$ in the first coordinate for the first time, which is also the point where $\lambda_2$ attains the value $c$ in the first coordinate for the first time.  Similarly, $\phi'$ attains the value $c$ in the first coordinate for the first time at $p'$, so $\pi_2(p')$ must also be the point where $\lambda_2$ attains the value $c$ in the first coordinate for the first time.  Hence $\pi_2(p) = \pi_2(p')$, which implies $p - p'$ is parallel to $e_2$.  Likewise $\pi_3(p) = \pi_3(p')$, which implies $p - p'$ is parallel to $e_3$.  Together these imply $p = p'$, which is a contradiction. 
\end{proof}

\section{Intuition}\label{sec:intuition}

\subsection{Fixed points and Degree} \label{subsec:degfp}

This section briefly reviews the topological degree of a map and its relevant properties \cite{hat}.  
The topological degree $\deg(f)$ of a map $f: \cir \to \cir$ is given by the number of times $f$ wraps around the circle.
That is, $\deg(f) = k$ when $f$ can be continuously deformed to the map $f_k(u) = (\sin(k\theta_u),\cos(k\theta_u))$ where $\theta_u$ is the positive angle between the vector $u$ and the vector $e_1 \in \cir$. 
Topological degree may alternatively be defined in terms of induced maps on homology. 
For a simple closed curve $\curvemap : \cir \to \R^d$ and a map $f : \curve \to \curve$, we let $\deg(f) = \deg(\curvemap^{-1} \circ f \circ \curvemap)$.

The two important properties of topological degree we use are: 
the degree of a composition of maps is the product of their degrees $\deg(f\circ g) = \deg(f)\deg(g)$, and if $f:\cir \to \cir$ does not have a fixed point then $\deg(f) = 1$.
Briefly, if $f$ does not have a fixed point, then 
\[
  f_t(x) = {\tilde f_t(x)}/{\|\tilde f_t(x)\|}, \text{ where } \tilde f_t(x) = (1-t)f(x) -tx,
\]
gives a continuous deformation from $f$ to the antipodal map. The antipodal map wraps once around the circle in the positive direction and therefore has degree $1$, so $f$ also has degree $1$.

In Section~\ref{sec:proof} we show how both of these properties can be generalized from maps on the circle to relations on the circle. 

\subsection{A Special Case}
Before we give a complete proof of Theorem \ref{thm:curve}, we present the basic ideas involved. To do this we assume that the simple closed curve is embedded in $\R^d$ in a very particular way.

\begin{theorem}[A special case of Theorem \ref{thm:curve}]\label{prop:2-to-1}
A simple closed curve in $\R^d$ cannot have three coordinate shadows such that each shadow is a simple path and the projection map to each shadow is $2$-to-$1$ on the interior of the path.
\end{theorem}

\begin{proof}
Suppose that $\curvemap:\cir \to \R^d$ is a simple closed curve such that $\pi_i(\curve)$ is a path for $i=1,2,3$ and that each $\pi_i$ restricted to $\curve$ is a $2$-to-$1$ map except for two points $a_i,\tilde a_i\in \curve$ which are mapped to the boundary of its shadow.

\begin{figure}
\centering
\includegraphics{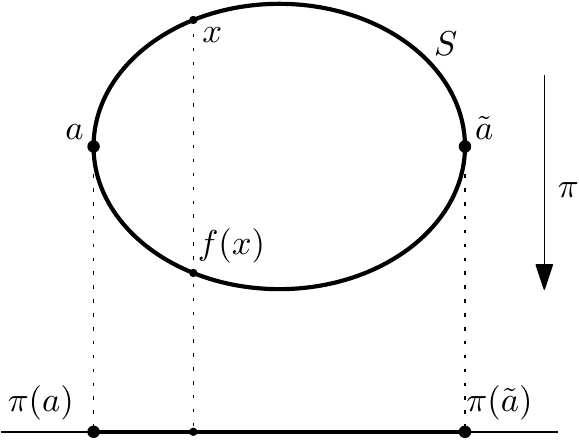}
\caption{The curve $\curve$ and the \emph{flipping} function $f$.}
\label{fig:proj}
\end{figure}

Then, there is a function $f_i:\curve \to \curve$ such that $a_i$ and $\tilde a_i$ are its only fixed points and $\pi_i(x)=\pi_i(f(x))$ for every $x \in \curve$ (see Figure~\ref{fig:proj}). This function has topological degree $-1$ and essentially ``{flips}'' $\curve$ in the $i$-th direction. 
Consequently, the degree of the composition $f_3\circ f_2\circ f_1$ is $-1$, and therefore has a fixed point $q_0\in \curve$.
Now consider the points
\[q_1=f_1(q_0),\quad q_2=f_2(q_1),\quad q_3=f_3(q_2)=q_0.\]
The three vectors $q_{i}-q_{i-1}$ for $i=1,2,3$ are respectively parallel to the first three elements of the canonical basis $e_1,e_2,e_3$ and add up to $0$.  Thus each of the vectors is $q_i-q_{i-1} = 0$, so $q_0=q_1=q_2$. This means that $q_0$ is a fixed point of each $f_i$ and is therefore mapped to an endpoint of each of the paths,
which is a contradiction by Lemma \ref{lem:endpoints}.
\end{proof}

\section{The General Case of Theorem \ref{thm:curve}}\label{sec:proof}

\subsection{Fixed points of Relations}

For Theorem~\ref{prop:2-to-1}, the assumption that the projection map is $2$-to-$1$ everywhere
in the interior of the path allows us to define the function
$f_i$ that ``flips'' the closed curve. 
In general, many points may be projected to a single point on the shadow, so instead of a map from the curve to itself, we obtain a relation between points on the curve.  
To prove Theorem \ref{thm:curve}, we adapt the argument in the previous section to relations. 

Recall that if $R \subset X\times Y$ and $R' \subset Y\times Z$ are the relations, their composition $R'\circ R$ and the inverse $R^{-1}$ are relations given by
\begin{align*}
R'\circ R &= \{(x,z):(x,y) \in R, (y,z) \in R'\text{ for some }y\in Y\},\\
R^{-1} &= \{(y,x) : (x,y) \in R\}.
\end{align*}
If $R\subset X\times X$ is a relation, a point $p\in X$ is a fixed point of $R$ when $(p,p)\in R$.
For a map 
\[ f = (f_1,\dots,f_n) : \cir \to (\cir \times \dots \times \cir),\] 
let $\deg (f) = (\deg (f_1), \dots, \deg(f_n))$.

We will make use of the fiber product of graphs in a manner similar to \cite{goodman1989mountain}.    
Given maps $f_i : X_i \to Y$ for $i \in \{1,2\}$, their fiber product is the set
\[ X_1 \times_{f_1,f_2} X_2 = \{(x_1,x_2) \in X_1 \times X_2 : f_1(x_1)=f_2(x_2) \}.\] 
When $f_1$ and $f_2$ are given by the same map $f$, respectively restricted to $X_1$ and $X_2$, we simply denote the fiber product by $X_1 \times_f X_2.$

If the domains of the maps $f_i : X_i \to Y$ are graphs $G_i = X_i$ ($1$-dimensional cell complexes), the target of the maps is either $Y = \mb{S}^1$ or $Y = \mb{R}^1$, 
and the $f_i$ is injective on edges, then we define the fiber product to also be a graph given by the following cell decomposition: a point $(x_1,x_2) \in G_1 \times_{f_1,f_2} G_2$ is a vertex when either $x_1$ is a vertex of $G_1$ or $x_2$ is a vertex of $G_2$.  
The compliment of the vertices is then the disjoint union of interiors of edges. 
See Figure~\ref{fig:graphs} for an example.

\begin{figure}
\centering
\includegraphics{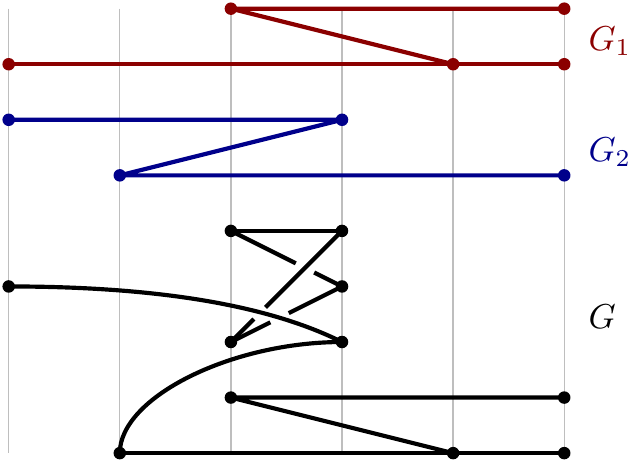}
\caption{The graphs $G_1$ and $G_2$ are given as planar drawing, the functions $f_1$ and $f_2$ are the projections in the vertical direction. Below is the fiber product $G = G_1 \times_{f_1,f_2} G_2$. In this representation the point $(x_1,x_2)\in G \subset G_1\times G_2$ is vertically aligned with the points $x_1 \in G_1$ and $x_2 \in G_2$.}
\label{fig:graphs}
\end{figure}

It is worth noting that the fiber product of two continuous surjections from the unit interval to itself with extrema fixed might not contain a path connecting the extrema. Such is the case in the example found in Figure~\ref{fig:paths}.

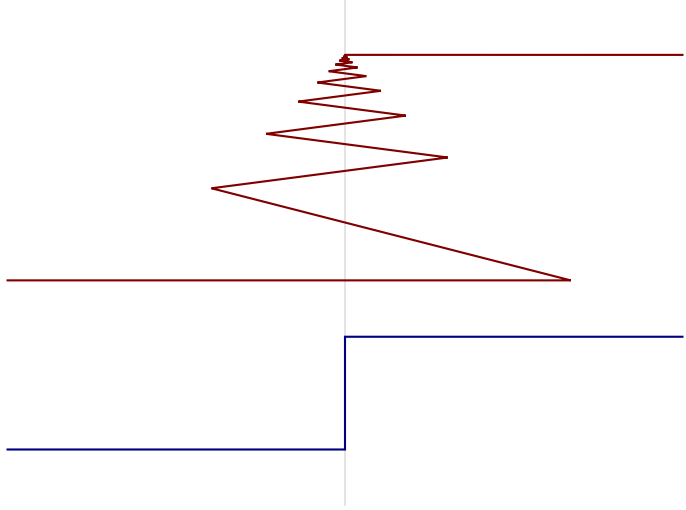
\begin{figure}
\centering
\begin{tikzpicture}[scale=1.5]

\draw[black!20]
(2,-2) -- (2,2.5)
;

\path
(0,0) coordinate (a)
(2,2) coordinate (b)
(4,0) coordinate (c)
;

\draw[red!50!black, thick] 
(-1,0) -- (c) (c)
\foreach \x in {1,...,15}
{ -- ($(a)!{1-1/1.3^(2*\x)}!(b)$)
($(a)!{1-1/1.3^(2*\x)}!(b)$) -- ($(c)!{1-1/1.3^(2*\x+1)}!(b)$) 
($(c)!{1-1/1.3^(2*\x+1)}!(b)$)
}
(b) -- (5,2)
;

\draw[blue!50!black, thick] 
(-1,-1.5) -- (2,-1.5) -- (2,-.5) -- (5,-.5)
;

\end{tikzpicture}
\caption{A pair of paths such that the product of their end points is not connected by a path in their fiber product.  Note that the upper path crosses the vertical line infinitely many times.}
\label{fig:paths}
\end{figure}

\begin{remark}\label{lem:graph_deg}
Given $G = G_1 \times_{f_1,f_2} G_2$ satisfying the above: the $G_i$ are graphs, the $f_i: G_i \to Y$ are injective on edges, and $Y\in\{\mb{S}^1,\mb{R}^1\}$;  
if $(x_1,x_2) = v \in G$ and $x_1 \in G_1$ are vertices but $x_2 \in G_2$ is not a vertex, then $v$ and $x_1$ have the same graphical degree.
\end{remark}

\begin{proof}
This follows from the observation that $v$ and $x_1$ have homeomorphic neighborhoods contained in their respective stars.  By the star of a vertex $v$, we mean the union of $v$ and its adjacent edges. 
To see this, consider an open interval $N_2 \subset G_2$ around $x_2$, and let $N_1$ be the connected component of $f_1^{-1}\circ f_2(N_2)$ that contains $x_1$.  
We may choose $N_2$ such that $N_1$ contains no vertices other than $x_1$.  Let $g_i$ be the restriction of $f_i$ to $N_i$ for $i\in\{1,2\}$.  Now the projection to the left factor gives a homeomorphism from the neighborhood $N_1 \times_{g_1,g_2} N_2$ of $v$ to the neighborhood $N_1$ of $x_1$.
\end{proof}

\begin{lemma}\label{lem:compose_rel}
Given relations $R_i \subset (\cir \times \cir)$ and curves $\phi_i : \cir \to R_i^{\circ}$ for $i\in\{1,2\}$ with $\deg(\phi_1) = (a,b_1)$ and $\deg(\phi_2) = (b_2,c)$, if both $b_1$ and $b_2$ are odd, then there is a curve $\phi : \cir \to (R_2\circ R_1)^{\circ}$ with $\deg(\phi) = (ka/b_1, kc/b_2)$ where $k$ is an odd common multiple of $b_1$ and $b_2$. 
\end{lemma}

For example, given
\[ \phi_1(\theta) = (\sin(\theta),\cos(\theta),\sin(3\theta),\cos(3\theta)), \]
\[ \phi_2(\theta) = (\sin(5\theta),\cos(5\theta),\sin(2\theta),\cos(2\theta)), \]
and $R_i$ a neighborhood of $\underline{\smash{\phi_i}}$,
we have $\underline{\smash{\phi}} \subset (R_2\circ R_1)^\circ$ where 
\[ \phi(\theta) = (\sin(5\theta),\cos(5\theta),\sin(6\theta),\cos(6\theta)). \]
Here $\deg(\phi_1) = (1,3)$, $\deg(\phi_2) = (5,2)$, and $\deg(\phi) = (5,6) = (k/3,2k/5)$ with $k=15$.

\begin{proof}[Proof of Lemma \ref{lem:compose_rel}]
We denote the factors of the curves $\phi_i : \mb{S}^1 \to \mb{T}^2$ by 
\[\phi_1=(\phi_{1\x},\phi_{1\y}) \text{ and } \phi_2 = (\phi_{2\y},\phi_{2\z}),\] 
and we denote their respective domains by $S_1$ and $S_2$. 
We may assume that the curves $\phi_i$ are straight edge drawings of cycle graphs $S_i$.  Otherwise replace $\phi_i$ with a straight edge approximation that has the same topological degrees and is  sufficiently close to $\phi_i$ to be contained in $R_i^\circ$.  We may further assume that the y-coordinate of the vertices of $\phi_1$ and $\phi_2$ are all distinct.
Let $G = S_1 \times_{\phi_{1\y},\phi_{2\y}} S_2$. 

Since the values of $\phi_{1\y}$ and $\phi_{2\y}$ are distinct at every vertex of $S_1$ and $S_2$, we have for each vertex $(s_1,s_2)$ of $G$ that $s_1$ and $s_2$ cannot both be vertices.
Hence by Remark~\ref{lem:graph_deg}, every vertex of $G$ has graphical degree $2$, so $G$ must be a union of disjoint cycles.  
Choose $y_0 \in \mb{S}^1$ that is not the coordinate of any vertex of $G$, and consider the fiber $F \subset G$ above $y_0$,
\[F = \{(s_1,s_2) \in S_1 \times S_2 : \phi_{1\y}(s_1) = \phi_{2\y}(s_2) = y_0 \}. \]

Since $b_1$ is odd, $\abs{\phi_{1\y}^{-1}(y_0)}$ is odd, and since $b_2$ is odd, $\abs{\phi_{2\y}^{-1}(y_0)}$ is odd, so $\abs{F} = \abs{\phi_{1\y}^{-1}(y_0)}\abs{\phi_{2\y}^{-1}(y_0)}$ is odd.  
Therefore, at least one of the cycles of $G$ must intersect the fiber $F$ in an odd number of points, and this cycle is the image of a simple closed curve $\sigma=(\sigma_1,\sigma_2) : \mb{S}^1 \to S_1 \times S_2$ that crosses $F$ an odd number of times.
There exists
\[ \xi : \cir \to \cir, \text{ such that }
\xi(t) = \phi_{1\y} \circ \sigma_1(t) = \phi_{2\y} \circ \sigma_2(t),\]
where the second equality holds since the range of $\sigma$ is in $G$.
Let $k$ be the topological degree of $\xi$.  The map $\xi$ crosses $y_0$ an odd number of times, so $k$ is odd. 

Recall that the topological degree of a composition of maps is the product of their degrees. 
Since $\deg(\phi_{1\y}) = b_1$ and $\deg(\phi_{1\y})\deg(\sigma_1) = \deg(\xi) = k$, we have $\deg(\sigma_1) = k/b_1$,
and similarly $\deg(\sigma_2) = k/b_2$.
Furthermore, since degree is integer valued, $k$ is a common multiple of $b_1$ and $b_2$.
Let 
\[\phi(t) = (\phi_{1\x} \circ \sigma_1(t), \phi_{2\z} \circ \sigma_2(t)). \] 
Observe that $\phi(t) \in (R_2 \circ R_1)^\circ$, since the range of $\sigma$ is in $G$. 
Since $\deg(\phi_{1\x}) = a$, and $\deg(\phi_{2\z}) = c$, 
we have $\deg(\phi) = (ka/b_1,kc/b_2)$.
\end{proof}

\begin{lemma}\label{lem:relfixedpoint}
For a curve $\phi : \cir \to \mb{T}^2$ with $\deg(\phi) = (a,b)$, if $a \neq b$, then $\underline{\smash{\phi}}$ intersects the diagonal $D = \{(u,u): u \in \mb{S}^1\} \subset \mb{T}^2$, and hence any relation $R$ containing $\underline{\smash{\phi}}$ has a fixed point. 
\end{lemma}

\begin{proof}
There is a deformation retraction $\rho_t$ from $\mb{T}^2 \setminus D$ to the curve $\psi(u) = (u,-u)$ with $\deg(\psi) = (1,1)$ given by
\[ \rho_t(u,v) = (u,\tfrac{w_t}{\|w_t\|}), \ \text{where} \ w_t = t(-u)+(1-t)v. \] 
If a curve $\psi' : \cir \to \mb{T}^2$ avoids $D$,  
then $\psi'$ can be deformed by $\rho$ to be in $\underline{\smash{\psi}}$, so $\deg(\psi')$ is a multiple of $\deg(\psi)$, which means $\deg(\psi') = (a,a)$ for some $a\in \mb{Z}$. 
Thus, any curve $\phi$ with $\deg(\phi) = (a,b)$ for $a \neq b$ cannot avoid the diagonal. 
\end{proof}

\subsection{Proof of Theorem \ref{thm:curve}}

Assume there is a simple closed curve $\curvemap : \cir \to \R^d$ such that the $i$th coordinate shadow $\pi_i(\curve) = \pathi{i}$ is a path $\pathimap{i}: \I \to \{x_i{=}0\} \subset \R^d$ for $i=1,2,3$,
and split $\curve$ into two closed arcs $T_i$ and $B_i$ (top and bottom) so that $\pi_i(T_i) = \pi_i(B_i) = \pi_i(\curve)$ and $T_i \cap B_i = \{a_i,\tilde a_i\}$, where $\pi_i(a_i)$ and $\pi_i(\tilde a_i)$ are the endpoints of $\pathi{i}$. 
We choose labels so that a closed curve $\cir \to \curve$ that first traverses $T_i$ from $a_i$ to $\tilde a_i$ then traverses $B_i$ from $\tilde a_i$ back to $a_i$ has topological degree $1$.

We define the relation $\Gamma_i \subset \curve \times \curve$ between a point in $T_i$ and a point in $B_i$ if they are mapped to the same point (see Figure \ref{fig:proj2}).
For $\eps > 0$ let 
\begin{align*}
\tilde\Gamma_{i,\eps} & = \left\{(t,b)\in T_i\times B_i: |\lambda_i^{-1}(\pi_i(t)) -\lambda_i^{-1}(\pi_i(b))| \leq \eps \right\},\\
N_\eps(p) &= \{(x,y) \in \curve^2: \|x-p\|+\|y-p\|\leq\eps \}, \\
\Gamma_{i,\eps} &= \tilde \Gamma_{i,\eps} \cup \tilde \Gamma_{i,\eps}^{-1} \cup N_\eps(a_i) \cup N_\eps(\tilde a_i). 
\end{align*}
Note that the $\Gamma_{i,\eps}$ are a nested family of compact sets with $ \cap_{\eps > 0} \Gamma_{i,\eps} = \Gamma_i$,
and that $a_i$ and $\tilde a_i$ are the only two fixed points of $\Gamma_i$.

To show that $\Gamma_3\circ\Gamma_2\circ\Gamma_1$ has a fixed point, we prove a series of claims below. 

\begin{figure}
\centering
\includegraphics{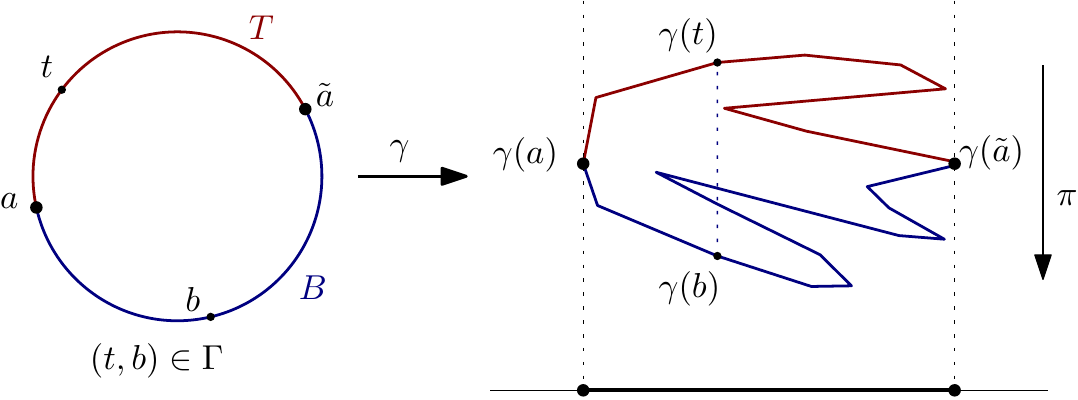}
\caption{The top and bottom arcs, and the relation $\Gamma$.}
\label{fig:proj2}
\end{figure}

\begin{claim}\label{clm:curve_i}
For $\eps > 0$ and for each $i\in\{1,2,3\}$, there is a curve $\psi_{i,\eps} : \mb{S}^1 \to \Gamma_{i,\eps}^\circ $ such that $\deg(\psi_{i,\eps}) = (1,-1)$.
\end{claim}

\begin{proof}

Since $\curve$ is contained in the curved surface $\pi_i^{-1}(\pathi{i}) = \pathi{i} +\R e_i$, we can factor $\curve$ through a map into the plane $\R^2$ by ``unraveling'' the surface.  Specifically, let 
\[ \eta_i : \curve \to (\I \times \R), \quad \eta_i(x) = (\pathimap{i}^{-1}\circ \pi_i(x), e^*_i( x) ). \]
Observe that for every $x,y\in \curve$,
\[\pi_i(x)=\pi_i(y) \quad\iff\quad e^*_1(\eta_i(x))=e^*_1(\eta_i(y)).\]

Let $\zeta_{i,\eps} : \curve \to (\I \times \R)$ such that $T'_{i,\eps} = \zeta_{i,\eps}(T_i)$ and $B'_{i,\eps} = \zeta_{i,\eps}(B_i)$ are polygonal paths with endpoints $a'_i = \eta_i(a_i)$ and $\tilde a'_i = \eta_i(\tilde a_i)$ that are sufficiently close to $\eta_i(T_i)$ and $\eta_i(B_i)$ that $G_{i,\eps} = T'_{i,\eps} \times_{e_1^*} B'_{i,\eps} \subset \eta_i(\tilde \Gamma_{i,\eps}^\circ)$. 
We may obtain such a $\zeta_{i,\eps}$, by choosing $T'_{i,\eps}$ and $B'_{i,\eps}$ to be polygonal $(\nicefrac{\eps}{2})$-approximations of $\eta_i(T_i)$ and $\eta_i(B_i)$, so that for all $(t,b) \in \zeta_{i,\eps}^{-1}(G_{i,\eps})$ with $y = e_1^*(\zeta_{i,\eps}(t)) = e_1^*(\zeta_{i,\eps}(b))$, we have
\begin{align*} 
|\lambda_i^{-1}(\pi_i(t)) - \lambda_i^{-1}(\pi_i(b))| &\leq | y - \lambda_i^{-1}(\pi_i(t))| + |y - \lambda_i^{-1}(\pi_i(b))| \\
&\leq \|\zeta_{i,\eps}(t) -\eta_i(t)\| + \|\zeta_{i,\eps}(b) -\eta_i(b)\|  \\
&< \eps.
\end{align*}
For simplicity, we choose $\zeta_{i,\eps}$ so that the first coordinate of the vertices of $T'_{i,\eps}$ and $B'_{i,\eps}$ are all distinct, except at the endpoints $a'_i$, $\tilde a'_i$. 

Since $a_i'$ has a single adjacent edge in $T'_{i,\eps}$ and a single adjacent edge in $B'_{i,\eps}$, the vertex $(a'_i,a'_i)$ of $G_{i,\eps}$ has graphical degree $1$.  Likewise $(\tilde a'_i,\tilde a'_i)$ has graphical degree $1$.  By Remark~\ref{lem:graph_deg}, every other vertex of $G_{i,\eps}$ has graphical degree $2$.  Hence, there is a path $\rho_{i,\eps} : \I \to G_{i,\eps}$ from $\rho_{i,\eps}(0) = (a',a')$ to $\rho_{i,\eps}(1) = (\tilde a',\tilde a')$.  

Let $\sigma_{i,\eps}: \I \to \tilde \Gamma_{i,\eps}^\circ$ be $\sigma_{i,\eps} = \eta_i^{-1} \circ \rho_{i,\eps}$, and define $\sigma': \I \to \tilde \Gamma_{i,\eps}^{-1\circ}$ by $\sigma'(t) = (y,x)$ where $(x,y) = \sigma(1-t)$, 
so that $\sigma$ traverses $\tilde \Gamma_{i,\eps}^\circ$ from $(a_i,a_i)$ to $(\tilde a_i, \tilde a_i)$ and 
$\sigma'$ traverses $\tilde \Gamma_{i,\eps}^{-1\circ}$ from $(\tilde a_i, \tilde a_i)$ to $(a_i,a_i)$.  Now we define $\psi_{i,\eps} : \mb{S}^1 \to \Gamma_{i,\eps}^\circ$ to be the closed curve that first follows $\sigma_{i,\eps}$ and then follows $\sigma'_{i,\eps}$; that is, 
\[
\psi_{i,\eps}(u) =
\begin{cases}
\sigma_{i,\eps} (\nicefrac{\theta_u}{\pi}) & \text{if }\theta_u \in [0,\pi] \\
\sigma'_{i,\eps} ((\nicefrac{\theta_u}{\pi})-1) & \text{if }\theta_u \in [\pi,2\pi]
\end{cases}
\]
where $\theta_u$ is the angle of the vector $u \in \cir$.

Observe that the first factor of $\psi_{i,\eps}$ traverses $T_i$ from
$a_i$ to $\tilde a_i$ and then traverses $B_i$ in the opposite direction, so
the first factor has topological degree $1$. Meanwhile, the second
factor of $\psi_{i,\eps}$ traverses $B_i$ from $a_i$ to $\tilde a_i$ and then
traverses $T_i$ in the opposite direction, so the second factor has
topological degree $-1$. Together this gives $\deg(\psi_{i,\eps}) = (1,-1)$.
\end{proof}

\begin{claim}\label{clm:epslast}
$\Gamma_{3,\eps} \circ\Gamma_{2,\eps} \circ\Gamma_{1,\eps}$ has a fixed point.
\end{claim}

\begin{proof}
Let $\psi_{i,\eps}$ be the curves as in Claim \ref{clm:curve_i} for $i=1,2,3$. By Lemma \ref{lem:compose_rel} applied to $\psi_{1,\eps}$ and $\psi_{2,\eps}$ there is a curve $\psi' : \cir \to (\Gamma_{2,\eps}\circ \Gamma_{1,\eps})^\circ$ such that $\deg(\psi') = (j,j)$ for some odd integer $j$, and by Lemma \ref{lem:compose_rel} applied to $\psi'$ and $\psi_{3,\eps}$ there is a curve $\psi_\eps : \cir \to (\Gamma_{3,\eps} \circ\Gamma_{2,\eps} \circ\Gamma_{1,\eps})^\circ$ such that $\deg(\psi_\eps) = (k,-k)$ for $k$ an odd multiple of $j$.
Therefore by Lemma~\ref{lem:relfixedpoint}, $\Gamma_{3,\eps} \circ\Gamma_{2,\eps} \circ\Gamma_{1,\eps}$ has a fixed point.
\end{proof}

\begin{claim}\label{clm:last}
$\Gamma_{3} \circ\Gamma_{2} \circ\Gamma_{1}$ has a fixed point.
\end{claim}

\begin{proof}
Let $q_{0,k}$ be a fixed point of $\Gamma_{3,\eps_k} \circ\Gamma_{2,\eps_k} \circ\Gamma_{1,\eps_k}$ for $\eps_k \to 0$ monotonically and $(q_{0,k},q_{1,k}) \in \Gamma_{1,\eps_k}$ and $(q_{1,k},q_{2,k}) \in \Gamma_{2,\eps_k}$ and $(q_{2,k},q_{0,k}) \in \Gamma_{3,\eps_k}$.  We may assume $q_{i,k} \to q_i$, since $\curve^2$ is compact, otherwise restrict to a convergent subsequence for each $i \in \{0,1,2\}$. 
Since the $\Gamma_{1,\eps}$ are nested, $(q_{0,j},q_{1,j}) \in \Gamma_{1,\eps_k}$ for $j \geq k$, and since $\Gamma_{1,\eps_k}$ is compact $(q_{0},q_{1}) \in \Gamma_{1,\eps_k}$.  Hence $(q_{0},q_{1}) \in \cap_{\eps_k} \Gamma_{1,\eps_k} = \Gamma_1$.  Similarly $(q_{1},q_{2}) \in \Gamma_{2}$ and $(q_{2},q_{0}) \in \Gamma_{3}$.
Thus, $q_0$ is a fixed point of $\Gamma_{3} \circ\Gamma_{2} \circ\Gamma_{1}$.
\end{proof}

Since $\Gamma_3\circ\Gamma_2\circ\Gamma_1$ has a fixed point, we have
a contradiction in exactly the same way as in the last two paragraphs
of the proof of Theorem \ref{prop:2-to-1} in Section \ref{sec:intuition}. 
Namely, the vectors $q_1-q_0$, $q_2-q_1$, $q_0-q_2$ are orthogonal to each other and sum to $0$, so each vector must be $0$.
This completes the proof of Theorem \ref{thm:curve}.

\section{Conclusion}

\subsection{Open Questions.}
The above questions can be generalized to higher dimensions in other
directions as well. 
\begin{enumerate}
\item\label{q1} 
What is the maximum number of coordinate shadows of a $k$-sphere embedded in $d$-space that can be contractible? 
\item\label{q2} 
What is the maximum number of coordinate shadows of a $k$-sphere embedded in $d$-space that can be embeddings of a $k$-ball? 
\item\label{q3} 
What is the maximum number of coordinate shadows of a $k$-ball embedded in $d$-space that can be embeddings of a $k$-sphere? 
\end{enumerate}

For $d = k+2$ the answer to question~\ref{q1} is $d$ by an inductive construction,  
and only special cases of Questions \ref{q2} and \ref{q3} for $k=1$ and $d=3$ were known
\cite{BDCD+2015}. 

\subsection{Difficulties of Generalization}

It is also worth noting a difficulty in extending to the case 
where $\cir$ is replaced with a higher dimensional sphere. 
If $\curvemap:\Ss^n\to\R^{d}$ is an embedding, 
then the preimage of $\partial \pi_i(\curve)$ might not separate $\Ss^n$ into multiple component.
Therefore, there does not seem to be a natural way to ``flip'' the sphere.

\section*{Acknowledgments}

This research was funded by CONACYT project 166306 and PAPIIT project IN112614. M. G. Dobbins was supported by the National Research Foundation of Korea NRF grant 2011-0030044, SRC-GAIA. H. Kim was supported by the Deutsche Forschungsgemeinschaft within the research training group `Methods for Discrete Structures' (GRK 1408). We are also thankful to Centro de Innovación Matemática (CINNMA) for all the support provided during this research.

\bibliographystyle{plain}
\bibliography{shadow}

\end{document}